\documentclass[oneside,12pt]{amsart}
\usepackage{amssymb, mathrsfs, amscd, mathabx}
\usepackage{tikz}
\usepackage[all]{xy}
\usepackage{enumerate, amsfonts, amsthm}

\usepackage{hyperref}
\hypersetup{
  pdfborder={0 0 0},
  colorlinks   = true,
  urlcolor     = blue,
  linkcolor    = blue,
  citecolor   = red
}

\usepackage{cleveref}

\usepackage{microtype} 
\usepackage{lmodern}   

\usepackage[a4paper,margin=1.1in]{geometry}

\newtheorem{thm}[subsection]{Theorem}
\newtheorem{cor}[subsection]{Corollary}
\newtheorem{prop}[subsection]{Proposition}
\newtheorem{lem}[subsection]{Lemma}

\newcommand{\del}[2]{{}}

\newcommand{\Z}{\mathbb Z}

\newcommand{\Q}{\mathbb Q}

\newcommand{\cN}{\mathcal N}
\newcommand{\lb}{b^{(2)}}

\newcommand{\cd}{\mathrm{cd}}

\title{Simple groups with strong fixed-point properties}

\author{Nansen Petrosyan}
\address{School of Mathematics, University of Southampton, Southampton SO17~1BJ, UK}
\email{n.petrosyan@soton.ac.uk}

\thanks{}


\date{\today}

\begin{document}

\begin{abstract}
We exhibit finitely generated torsion-free groups for which any action on any finite-dimensional CW-complex with finite Betti numbers has a global fixed point.
\end{abstract}

\maketitle

\section{Introduction}

In \cite{FFS}, Fournier-Facio and Sun introduced a method of constructing, so called, relative torsion-free Tarski monsters with strong homological control. 
In this note, we apply their construction to obtain examples of groups with extreme fixed-point properties.

We will say that a group $G$ is a {\it torsion-free Tarski monster} if $G$ is a torsion-free infinite simple group all of whose proper subgroups are cyclic. 

\begin{thm}{\label{main}}There are uncountably many pairwise non-isomorphic finitely generated torsion-free Tarski monsters $G$ such that any action of $G$ on any finite dimensional  CW-complex  with finite Betti numbers has a global fixed point. 

\end{thm}

As usual, here and throughout, Betti numbers are taken over $\mathbb Q$.
Since finite domination,  acyclicity and homotopy equivalence to a compact manifold or finite complex all imply finiteness of Betti numbers, \Cref{main} applies to a broad class of topological spaces commonly studied in geometric topology and group actions.

The study of fixed-point properties for group actions on such spaces has a rich history. A classical result of Smith theory \cite{PASmith41} shows that a finite $p$-group acting on a finite-dimensional mod-$p$ acyclic CW-complex must have a fixed point. The first infinite groups with similar fixed-point properties were constructed in \cite{arz09}. Since these constructions rely on torsion, they do not apply to torsion-free groups.  More recently, the first torsion-free examples with the property that any action on a finite dimensional contractible CW-complex has a global fixed point were constructed in  \cite[Cor.~C]{FFS}. 

It is worth emphasising that in \Cref{main}, neither the assumption of finite-dimensionality nor that of finite Betti numbers can be entirely omitted. For instance, any group acts freely on itself, regarded as a $0$-dimensional complex. Moreover, even for simply connected complexes, the omission of either assumption leads to free actions, since any discrete group acts freely on the universal cover of its presentation 2-complex and on the universal cover of its classifying space.

In contrast to \Cref{main},  the next result shows that one can construct torsion-free Tarski monsters of arbitrarily large but finite cohomological dimension, which therefore admit proper actions on finite dimensional contractible CW-complexes. Yet,  any action of such a group on a complex of sufficiently smaller dimension with finite Betti numbers must still have a global fixed point. This phenomenon mirrors the behaviour observed in earlier examples of groups with fixed-point properties that depend on the dimension of the space; see, for instance, \cite[Thm.~1.7]{arz09}, \cite{serre73, serre77}, \cite{farb09}, \cite{BV11}, \cite{bridson12}, and \cite{varghese14}. 

\begin{thm}{\label{prop}} For each integer $n\geq 4$, there is a finitely generated torsion-free Tarski monster $G_n$ with $\cd (G_n)=2n$ such that for any $m\leq n-4$,  any action of $G_n$ on any $m$-dimensional CW-complex  with finite Betti numbers has a global fixed point. 
\end{thm}

The main idea of the proofs of \Cref{main} and \Cref{prop} is to pit  the two spectral sequences associated with an action of the group~$G$ on a CW-complex against each other, by comparing their von Neumann dimensions. Since $G$ arises as a quotient of certain even-dimensional hyperbolic arithmetic lattices, its $L^2$-Betti numbers can be approximated by those of the lattices, whose spectra have large gaps which we exploit.

\subsection*{Acknowledgments.} 
We are grateful to Francesco Fournier-Facio, Ian Leary, and Ashot Minasyan for reading an earlier version of this note and for their helpful suggestions. We also thank the referee for a careful reading and for many useful comments.

\section{Background}

\subsection{Homological finiteness properties}
A {\it $G$-CW-complex} is a CW-complex $X$ equipped with a $G$-action that is compatible with the cell structure, in the sense of L\"uck \cite[Def.~1.25]{luck2002l2}.  This notion is equivalent to that of an admissible $G$-complex in the terminology of Brown \cite[Chap.~IX, §10]{brown1982cohomology}, characterised by the property that any element of $G$ which stabilises a cell $\sigma \subseteq X$ must fix $\sigma$ pointwise.  Throughout, we will assume that all $G$-actions are of this form.

We say that $X$ is a \textit{free $G$-CW complex} if the action $G$ on $X$ is free. 
By $EG$ we mean a contractible free $G$-CW complex, and by $BG$ we mean the quotient $EG/G$ called a {\it classifying space} of $G$. If we do not specify the choice of $EG$, then we simply choose an arbitrary contractible free $G$-CW complex as $EG$. Such a complex always exists and any two choices for $EG$ are $G$-homotopy equivalent \cite{luck05_survey}.  A group $G$ is said to be of {\it type} $F_n$ if there is a $BG$ with finite $n$-skeleton and it is of {\it type} $F$ if $BG$ is finite.

The \textit{cohomological dimension of $G$} can be defined by
$$\cd (G)=\sup\{n\in\mathbb{N}\mid  H^{n}(G; A)\neq 0\text{ for some } \Z G\text{-module }A\}.$$

\subsection{$L^2$-Betti numbers} For details, we refer to \cite{luck2002l2, luck2003l2}. Let $G$ be a discrete group. The \textit{group von Neumann algebra} of $G$, denoted $\cN(G)$, is the set of all bounded  linear operators on $\ell^2(G)$ that commute with the left regular representation of $G$.  The \textit{von Neumann dimension $\dim_{\cN(G)}$} of $G$ is a function from the set of $\cN(G)$-modules to $[0,\infty]$ as defined in \cite[Chapter 6]{luck2002l2}, satisfying the following properties.

\begin{prop}\label{prop. properties of dimension} Let $G$ be a group. The following properties hold.
\begin{enumerate}[(i)]
    \item\label{item. finite dimension} If $A$ is a finitely generated $\cN(G)$-module, then $\dim_{\cN(G)} A<\infty$.

    \item\label{item. only depend on iso class} If $A$ and $B$ are isomorphic $\cN(G)$-modules, then $\dim_{\cN(G)} A=\dim_{\cN(G)} B$.

    \item\label{item. additivity} If $0\rightarrow A\rightarrow B \rightarrow C\rightarrow 0$ is an exact sequence of $\cN(G)$-modules, then 
    $$\dim_{\cN(G)} B=\dim_{\cN(G)} A+\dim_{\cN(G)} C,$$
    where the laws of the summation of $[0,\infty)$ extends naturally to $[0,\infty]$.

    \item\label{item. cofinal} Let $\{A_i\mid i\in I\}$ be a cofinal system of $\cN(G)$-submodules of an $\cN(G)$-module $A$, i.e., $A=\bigcup_{i\in I}A_i$ and for every two indices $i,j\in I$ there exists an index $k\in I$ such that $A_i, A_j\subset A_k$. Then $\dim_{\cN(G)} A=\sup\{\dim_{\cN(G)} A_i\mid i\in I\}$. 
\end{enumerate}
    
\end{prop}

Suppose $\theta : K\to G$ is a homomorphism of groups. It induces a ring homomorphism from $\Z K$ to $\Z G$. Since $\Z G$  naturally embeds in $\cN(G)$, the ring $\cN(G)$ becomes a left $\Z K$-module.   So, we can form the  chain complex $C_\ast(EK)\otimes_{\Z K}\cN(G)$ of right $\cN(G)$-modules and denote its homology groups by $H_\ast(K ; \cN(G))$ and by 
$$\lb_\ast (K\xrightarrow{\theta} G):= \dim_{\cN(G)}  H_\ast(K ; \cN(G)), \; \; \lb_\ast (G):= \lb_\ast (G\xrightarrow{id} G).$$

We will only be interested in the special cases where $\theta$ is either a quotient homomorphism or the inclusion homomorphism of a subgroup.

 Jaikin-Zapirain and L\'{o}pez-\'{A}lvarez in \cite{jaikin2020strong}  proved L\"uck's Approximation Conjecture for virtually locally indicable groups.  We will need  a quantitative application  of this result from \cite{PS25}, which is  similar to \cite[Thm.~2.4]{FFS}.

\begin{cor}[{\cite[Cor.~3.4]{PS25}}]\label{cor:betti_approx}
    Let $G$ be a finitely generated virtually locally indicable group of type $F_{p+1}$ for some $p\in\mathbb N$. Then for every $\delta>0$, there exists a finite subset $\mathcal F_{p,\delta}\subset G\smallsetminus\{1\}$ such that if a normal subgroup $N\lhd G$ satisfies $N\cap \mathcal F_{p,\delta}=\emptyset$, then for all $n\leq p$,
    \[|b^{(2)}_n(G)-b^{(2)}_n(G\twoheadrightarrow G/N)|<\delta.\]
\end{cor}

\section{Proof of  Main Results}

A key input for the proofs is Theorem 4.1 of \cite{FFS}, of which we state only the parts necessary for our arguments.

\begin{thm}[{Fournier-Facio--Sun, \cite[Thm.~4.1]{FFS}}]\label{monster} For each positive integer $k$, let $L_k$ be a torsion-free non-elementary hyperbolic group, and $B_k \subset L_k$ be finite subsets. Then there exists a group $M$ such that 
\begin{enumerate}[(i)]
\item $M$ is finitely generated torsion-free simple group.

\item Each proper subgroup of $M$ is cyclic.

\item There is an epimorphism $\pi_k: L_k\twoheadrightarrow M$ that is injective on $B_k$.

\item \label{item:homol_iso} For every $\Z M$-module (\,$\Z M$--$\cN(M)$-bimodule) $A$, and $i\geq 3$, there is an isomorphism (of right $\cN(M)$-modules)
$$H_i(M; A)\cong \bigoplus_{k\geq 1} H_i(L_k ; A),$$
induced by the epimorphism $\Asterisk_{k\geq 1} L_k\twoheadrightarrow M$ whose restriction to each subgroup $L_k$ is $\pi_k$.
\end{enumerate} 
\end{thm}



\begin{lem}\label{lemma_2} For each  integers $n\geq 1$ and $k\geq 1$, there is a torsion-free virtually locally indicable uniform arithmetic lattice $L_{n,k}$ in $SO(2n, 1)$ of type $F$, such that
\begin{enumerate}[(i)]
\item $\cd(L_{n,k})=2n,$

\item $\lb_{n} (L_{n,k})> k,$

\item $\lb_j (L_{n,k})=0, \; \forall j\ne n$.
\end{enumerate} 
\end{lem}

\begin{proof} For each $n\geq 1$, let $L_n$ be a torsion-free uniform arithmetic lattice of simplest type in $SO(2n, 1)$. Since $\mathbb H^{2n}/{L_n}$ is a closed aspherical manifold, it follows that $L_n$ is of type $F$ and $\cd (L_n) =2n$. By Theorem 2.3 of \cite{jost2000vanishing}, $\lb_j(L_n)=0$ if $j\ne n$  and $\lb_n(L_n)=|\chi(L_n)|$ which is nonzero (see e.g.~\cite[Theorem 7.5]{PS25}). By Theorem 1.8 of \cite{bergeron2012boundary}, each $L_n$ is virtually compact special and therefore virtually locally indicable (see e.g.\,\cite[Prop.~4.12(iii)]{PS25}).

Now, for each $k\geq 1$, let $L_{n,k}$  be a finite index subgroup of $L_n$ such that $\lb_n(L_{n,k})=|\chi(L_{n,k})|>k$. By  \cite[Thm.~2.3]{jost2000vanishing}, $\lb_j(L_{n,k})=0$ if $j\ne n$.
\end{proof}

\begin{proof}[Proof of \Cref{main}] Given an infinite sequence $S=\{n_l \; | \; n_{l+1}> n_l +l +1 \}\subset \mathbb N_{\geq 3}$ and $k\geq 1$, let $L_{n_l,k}$ be hyperbolic lattice groups as in Lemma \ref{lemma_2}. For each pair $(n_l, k)$, let $B'_{l,k}\subset L_{n_l,k}\smallsetminus \{1\}$ be the finite set given by \Cref{cor:betti_approx} with respect to $p=2n_l$ and $\delta=2^{-n_l\cdot k}$.  We apply Theorem \ref{monster} with  $\{L_{n_l,k}\}$ for all $l \in \mathbb{N}$ and all $k \geq 1$, and wth the corresponding sets $B_{l,k}:=B'_{l,k}\cup \{1\}$, to obtain a finitely generated torsion-free simple group $M_S$ such that $\cd (M_S)=\infty$. 


By \Cref{monster}(\ref{item:homol_iso}), Lemma \ref{lemma_2} and \Cref{cor:betti_approx},
\begin{align*}
 \forall j\in S, \; \lb_{j}(M_S)&=\sum_{k\geq 1, n_l \in S} \lb_{j}(L_{{n_l}, k}\twoheadrightarrow M_S)\\
&\geq \sum_{k\geq 1} \lb_{j}(L_{{j}, k}\twoheadrightarrow  M_S)\\
&\geq \sum_{k\geq 1} (\lb_{j}(L_{{j}, k})-2^{-j\cdot k})=\infty;\\
 \forall j\notin S, j\geq 3, \; \lb_j(M_S)&=\sum_{k\geq 1, n_l \in S} \lb_{j}(L_{{n_l}, k}\twoheadrightarrow  M_S)\\
&\leq \sum_{k\geq 1, n_l \in S}(\lb_{j}(L_{{n_l},k})+2^{-n_l\cdot k})\\
&\leq \sum_{k\geq 1, n_l \in S} 2^{-n_l\cdot k}\leq 1.
\end{align*}
This shows that if $S\ne S'$, then $\{\lb_*(M_S)\}\ne \{\lb_*(M_{S'})\}$ and hence, $M_S\ncong M_{S'}$. Since the collection of the above subsets $S$ is uncountable, it  follows that, up to isomorphism, the family of the corresponding groups $M_S$ is uncountable.

Suppose that for some $S$, the group $M_S$ acts on an $m$-dimensional CW-complex $X$ such that $H_i(X; \Q)$ has finite dimension $b_i$ over $\Q$ for each $i$. 
The action of $M_S$ on the complex $X$ induces an action on $H_i(X; \Q)$ which may be seen as a representation $\rho_i:M_S\to \mbox{Aut}(H_i(X;\Q))=\mathrm{GL}(b_i, \Q)$. By a theorem of Mal'cev \cite{Malcev40}, any finitely generated subgroup of $\mathrm{GL}(b_i, \Q)$ is residually finite. Since $M_S$ is simple, it follows that $\rho_i$ is trivial for all $i$, i.e.~$M_S$ acts trivially on the rational homology of $X$.

Suppose, by a way of contradiction, that the action of $M_S$ on $X$ does not have a  fixed point.  

Next, we recall two standard spectral sequences (see e.g.\, \cite[\S VII.7]{brown1982cohomology}) associated to an action of a group on a CW-complex, adapted to our setting. We take coefficients in the ring $\cN(M_S)$, which is a $\Z M_S$--$\cN(M_S)$-bimodule. This ensures that the spectral sequences naturally have the structure of right $\cN(M_S)$-modules.

Consider the double complex of right $\cN(M_S)$-modules 
$$P_* \otimes_{\Z M_S} \displaystyle{(C_*(X;\Q)\otimes_{\Q} \cN(M_S))},$$
where $P_*$ is a projective resolution of $\Z$ over $\Z M_S$ and $C_*(X; \Q)$ is the cellular chain
complex of $X$ with rational coefficients. Let $\mbox{Tot}_*$ denote the total complex of this double complex. The double complex gives the $E^1$-term spectral sequence (see \cite[\S VII.5 (5.3) p.\,169]{brown1982cohomology}),
$$E_{p,q}^1=H_q(M_S; C_p(X;\Q) \otimes_{\Q} \cN(M_S)) \Longrightarrow
H_{p+q}(\mbox{Tot}_*),$$
Since 
$$C_p(X;\Q) \otimes_{\Q} \cN(M_S)\cong \displaystyle{\bigoplus_{\sigma \in \Sigma_p}} \Q M_S \otimes_{\Q M_{\sigma}}  \cN(M_S),$$
as right $\cN(M_S)$-modules, where $X_p$ is the collection of all the $p$-cells and $\Sigma_p$ denotes
a set of representatives of all the $M_S$-orbits in $X_p$, and $M_{\sigma}$ is the stabiliser of $\sigma$ , 
the spectral sequence transforms to
$$E_{p,q}^1=\displaystyle{\bigoplus_{\sigma \in \Sigma_p}}H_q(M_{\sigma}; \cN(M_S))\Longrightarrow
H_{p+q}(\mbox{Tot}_*),$$
\noindent (see \cite[\S VII.7 (7.7) p.\,173]{brown1982cohomology}). According to our hypotheses, each stabiliser is a proper subgroup of $M_S$ and hence is cyclic. So, by Theorem 6.37 and Theorem 6.54 (7) of \cite{luck2002l2},  $\lb_q(M_{\sigma}\hookrightarrow M_S)=0$ for all $q>0$. Therefore, $\dim_{\cN(M_S)} E_{p,q}^1 = 0$
when $p>m$ or $q>0$. This implies that   $\mbox{dim}(H_{*}(\mbox{Tot}_*))\leq m$,
where $$\mbox{dim}(H_{*}(\mbox{Tot}_*)):=\max \{ i \; | \;  \dim_{\cN(M_S)}H_{i}(\mbox{Tot}_*)\ne 0\}.$$

The double complex also gives the $E^2$-term spectral sequence
$$E_{p,q}^2= H_p(M_S; H_q(X;\Q)\otimes_{\Q} \cN (M_S))
\Longrightarrow H_{p+q}(\mbox{Tot}_*),$$
(see \cite[\S VII.7 (7.2) p.\,172]{brown1982cohomology}).
Observe that $E_{p,q}^2 = 0$ when $q> m$. Denote by $ \cN (M_S)^{b_q}$ the direct sum of $b_q$ copies of $\cN (M_S)$, equipped with the diagonal $\Z M_S$--$\cN (M_S)$-bimodule structure.  Since $M_S$ acts trivially on $H_*(X; \Q)$, 
$$H_q(X;\Q)\otimes_{\Q} \cN (M_S)\cong \cN (M_S)^{b_q}$$
as $\Z M_S$--$\cN (M_S)$-bimodules. The
spectral sequences then becomes
$$E_{p,q}^2= H_p(M_S; \cN (M_S))^{b_q}
\Longrightarrow H_{p+q}(\mbox{Tot}_*).$$

\begin{figure}[ht]
    \centering
    \begin{tikzpicture}[scale=1.2,>=stealth]


    \draw[dashed] (0,1)  -- (7,1);
    \node[left] at (0,1) {$m$};
    
 \fill[lightgray!30] (0,0) rectangle (7,1);

    \draw (2,0) -- (2,1);
    \draw (5,0) -- (5,1);

    \node[below] at (2,0) {$n_{m}$};
    
    \node[below] at (5,0) {$n_{m+1}$};
    

     \draw[->] (0,0) -- (7,0) node[right] {$p$};
    \draw[->] (0,0) -- (0,2) node[above] {$q$};

    \draw[->] (5,0) -- (3, 0.6);

    \node at (2.7, 0.3) {$d^{i\leq m+1}$};
    
     \node at (2.7, 1.6) {$d^{i> m+1}$};

    \draw[->] (5,0) -- (1.8, 1.5);
    
 
    \fill (2,0) circle (1pt);
    \fill (5,0) circle (1pt);

    \end{tikzpicture}
    \caption{Shaded region $0\leq q\leq m$ illustrates all possible non-zero terms on the $E^2$-page and higher pages. The arrow representing the differential $d^i$ starting at $(n_{m+1}, 0)$ remains within this region precisely when $i\leq m+1$, in which case its terminal point does not reach the vertical line at $n_m$.}
    \label{fig:spectral_sequence_1}
\end{figure}
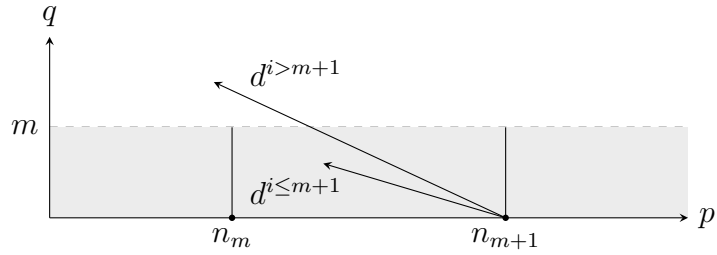

Let $n_{m}, n_{m+1}\in S$. Then, $\dim_{\cN(M_S)} E^2_{n_{m+1}, 0}=b_0\cdot \lb_{n_{m+1}}(M_S)=\infty$ and if $n_m< p<n_{m+1}$, then 
\begin{align*}
\dim_{\cN(M_S)} E^2_{p, q}&=\dim_{\cN(M_S)} H_p(M_S; \cN (M_S))^{b_q}, \\
&=b_q\cdot \lb_p(M_S) \leq  b_q.
\end{align*}
By \Cref{prop. properties of dimension}(\ref{item. additivity}), this  implies that $\dim_{\cN(M_S)} E^i_{p, q}$ is bounded for all $i\geq 2$ whenever $n_m< p<n_{m+1}$. Since $E_{p,q}^2 = 0$ for $q> m$ and by construction, $n_{m+1}-n_m> m+1$, the target of the differential $d^i_{n_{m+1}, 0}$ has a finite von Neumann dimension for all $i\geq 2$ (see \Cref{fig:spectral_sequence_1}). By, again invoking \Cref{prop. properties of dimension}(\ref{item. additivity}), for each $i\geq 2$, we obtain
$$\dim_{\cN(M_S)} E^{i+1}_{n_{m+1}, 0}= \dim_{\cN(M_S)}  \ker d^i_{{n_{m+1}},0}=\infty.$$
Since $E_{n_{m+1}, 0}^{\infty} =E_{n_{m+1}, 0}^{m+2}$, we obtain $\dim_{\cN(M_S)} E_{n_{m+1}, 0}^{\infty} =\infty$. Hence, $\mbox{dim}(H_{*}(\mbox{Tot}_*))\geq n_{m+1}$ which contradicts our previous estimate.
\end{proof}

\begin{proof}[Proof of Theorem \ref{prop}] The proof is similar to the proof of Theorem \ref{main}. So, we only mention the key differences and outline the rest.

Fix $n\geq 4$.  For each $k\geq 1$, let $L_{k}:=L_{n,k}$ be hyperbolic lattice groups as in Lemma \ref{lemma_2} and  let $B'_{k}\subset L_{k}\smallsetminus \{1\}$ be the finite set given by \Cref{cor:betti_approx} with respect to $p=2n$ and $\delta=2^{-k}$.  We apply Theorem \ref{monster} with  $\{L_{k}\}$ and $B_{k}:=B'_k\cup \{1\}$, to obtain a finitely generated torsion-free simple group $M$ such that $\cd (M)=2n$.

By Theorem \ref{monster}(\ref{item:homol_iso}), Lemma \ref{lemma_2} and \Cref{cor:betti_approx},
\begin{align*}
\lb_{n}(M)&=\sum_{k=1}^{\infty} \lb_{n}(L_k\twoheadrightarrow  M)\\
&\geq \sum_{k=1}^{\infty} (\lb_{n}(L_k)-2^{-k})=\infty, \\
\lb_j(M)&=\sum_{k=1}^{\infty} \lb_j(L_k\twoheadrightarrow  M),\\
&\leq \sum_{k=1}^{\infty} (\lb_j(L_k)+2^{-k})\leq 1, \; \forall j\ne n, j\geq 3.
\end{align*}

Suppose $M$ acts on an $m$-dimensional CW-complex $X$ such that $m\leq n-4$ and $H_i(X;\Q)$ has finite dimension $b_i$ over $\Q$ for all $i$.  By Mal'cev's Theorem, any finitely generated subgroup of $\mathrm{GL}(b_i, \Q)$ is residually finite. Thus, the action of $M$ on $X$ induces the trivial action on the rational homology of $X$.

Suppose, by a way of contradiction, that $M$ acts without a global fixed point on $X$.  Consider the double complex 
$$P_* \otimes_{\Z M} \displaystyle{(C_*(X; \Q)\otimes_{\Q} \cN(M))},$$
where $P_*$ is a projective resolution of $\Z$ over $\Z M$. It gives the $E^1$-term spectral sequence
$$E_{p,q}^1=\displaystyle{\bigoplus_{\sigma \in \Sigma_p}}H_q(M_{\sigma}; \cN(M))\Longrightarrow
H_{p+q}(\mbox{Tot}_*).$$
According to our hypotheses, each stabiliser is a proper subgroup of $M$ and hence is cyclic. Therefore, $\dim_{\cN(M_S)} E_{p,q}^1 = 0$
when $p>m$ or $q>0$. This implies that  $\mbox{dim}(H_{*}(\mbox{Tot}_*))\leq m$.

The double complex also gives the $E^2$-term spectral sequence
$$E_{p,q}^2= H_p(M; \cN (M))^{b_q}
\Longrightarrow H_{p+q}(\mbox{Tot}_*).$$

\noindent It follows that $\dim_{\cN(M)} E^2_{n, 0}=b_0\cdot \lb_{n}(M)=\infty$ and if $p\ne n, \: p\geq 3$, then 
\begin{align*}
\dim_{\cN(M)} E^2_{p, q}&=\dim_{\cN(M)} H_p(M; \cN (M))^{b_q}, \\
&=b_q\cdot \lb_p(M) \leq  b_q.
\end{align*}
This implies that if $p\ne n, \: p\geq 3$, then $\dim_{\cN(M)} E^i_{p, q}$ is bounded for $q\leq m$ and vanishes for $q> m$. Note that if $i\leq m+1$, then $n-i\geq n-m-1\geq 3$, because by assumption $m\leq n-4$. So, $\dim_{\cN(M)} E^i_{n-i, i-1}$ is bounded for $i\leq m+1$ and vanishes for $i> m+1$. 
Since $d^i_{{n},0}: E^i_{n, 0}\to E^i_{n-i, i-1}$, we obtain
$$\dim_{\cN(M)} E^{i+1}_{n, 0}= \dim_{\cN(M)}  \ker d^i_{{n},0}=\infty.$$
Thus, $\dim_{\cN(M)} E_{n, 0}^{\infty} =\infty$ and  hence, $\mbox{dim}(H_{*}(\mbox{Tot}_*))\geq n$, which contradicts our previous estimate.
\end{proof}

\bibliographystyle{alpha}
\bibliography{petrosyan_refs}

\newcommand{\etalchar}[1]{$^{#1}$}
\begin{thebibliography}{ABJ{\etalchar{+}}09}

\bibitem[ABJ{\etalchar{+}}09]{arz09}
G.~Arzhantseva, M.~R. Bridson, T.~Januszkiewicz, I.~J. Leary, A.~Minasyan, and
  J.~\'Swi\c~atkowski.
\newblock Infinite groups with fixed point properties.
\newblock {\em Geom. Topol.}, 13(3):1229--1263, 2009.

\bibitem[Bri12]{bridson12}
M.~R. Bridson.
\newblock On the dimension of {CAT}(0) spaces where mapping class groups act.
\newblock {\em J. Reine Angew. Math.}, 673:55--68, 2012.

\bibitem[Bro94]{brown1982cohomology}
{K}. Brown.
\newblock {\em Cohomology of groups}, volume~87 of {\em Graduate Texts in
  Mathematics}.
\newblock Springer-Verlag, New York, 1994.
\newblock Corrected reprint of the 1982 original.

\bibitem[BV11]{BV11}
M.~R. Bridson and K.~Vogtmann.
\newblock Actions of automorphism groups of free groups on homology spheres and
  acyclic manifolds.
\newblock {\em Comment. Math. Helv.}, 86(1):73--90, 2011.

\bibitem[BW12]{bergeron2012boundary}
N.~Bergeron and D.~Wise.
\newblock A boundary criterion for cubulation.
\newblock {\em Amer. J. Math.}, 134(3):843--859, 2012.

\bibitem[Far09]{farb09}
B.~Farb.
\newblock Group actions and {H}elly's theorem.
\newblock {\em Adv. Math.}, 222(5):1574--1588, 2009.

\bibitem[FFS25]{FFS}
F.~Fournier-Facio and B.~Sun.
\newblock Dimensions of finitely generated simple groups and their subgroups.
\newblock {\em arXiv preprint arXiv:2503.01987}, 2025.

\bibitem[JX00]{jost2000vanishing}
J.~Jost and Y.~Xin.
\newblock Vanishing theorems for {$L^2$}-cohomology groups.
\newblock {\em J. Reine Angew. Math.}, 525:95--112, 2000.

\bibitem[JZLA20]{jaikin2020strong}
A.~Jaikin-Zapirain and D.~L\'{o}pez-\'{A}lvarez.
\newblock The strong {A}tiyah and {L}\"{u}ck approximation conjectures for
  one-relator groups.
\newblock {\em Math. Ann.}, 376(3-4):1741--1793, 2020.

\bibitem[L{\"u}c02]{luck2002l2}
W.~L{\"u}ck.
\newblock {\em {$L^2$}-invariants: theory and applications to geometry and
  {$K$}-theory}, volume~44 of {\em Ergebnisse der Mathematik und ihrer
  Grenzgebiete. 3. Folge. A Series of Modern Surveys in Mathematics [Results in
  Mathematics and Related Areas. 3rd Series. A Series of Modern Surveys in
  Mathematics]}.
\newblock Springer-Verlag, Berlin, 2002.

\bibitem[L{\"u}c05]{luck05_survey}
W.~L{\"u}ck.
\newblock Survey on classifying spaces for families of subgroups.
\newblock In {\em Infinite groups: geometric, combinatorial and dynamical
  aspects}, volume 248 of {\em Progr. Math.}, pages 269--322. Birkh\"{a}user,
  Basel, 2005.

\bibitem[L{\"u}c09]{luck2003l2}
W.~L{\"u}ck.
\newblock {$L^2$}-invariants from the algebraic point of view.
\newblock In {\em Geometric and cohomological methods in group theory}, volume
  358 of {\em London Math. Soc. Lecture Note Ser.}, pages 63--161. Cambridge
  Univ. Press, Cambridge, 2009.

\bibitem[Mal40]{Malcev40}
A.~Malcev.
\newblock On isomorphic matrix representations of infinite groups.
\newblock {\em Rec. Math. [Mat. Sbornik] N.S.}, 8/50:405--422, 1940.

\bibitem[PS25]{PS25}
N.~Petrosyan and B.~Sun.
\newblock ${L}^2$-{B}etti numbers of {D}ehn fillings.
\newblock {\em arXiv preprint arXiv:412.16090}, 2025.

\bibitem[Ser74]{serre73}
J.-P. Serre.
\newblock Amalgames et points fixes.
\newblock In {\em Proceedings of the {S}econd {I}nternational {C}onference on
  the {T}heory of {G}roups ({A}ustralian {N}at. {U}niv., {C}anberra, 1973)},
  volume Vol. 372 of {\em Lecture Notes in Math.}, pages 633--640. Springer,
  Berlin-New York, 1974.

\bibitem[Ser77]{serre77}
J.-P. Serre.
\newblock {\em Arbres, amalgames, {${\rm SL}\sb{2}$}}, volume No. 46 of {\em
  Ast\'erisque}.
\newblock Soci\'et\'e{} Math\'ematique de France, Paris, 1977.
\newblock Avec un sommaire anglais, R\'edig\'e{} avec la collaboration de Hyman
  Bass.

\bibitem[Smi41]{PASmith41}
P.A. Smith.
\newblock Fixed-point theorems for periodic transformations.
\newblock {\em Amer. J. Math.}, 63:1--8, 1941.

\bibitem[Var14]{varghese14}
O.~Varghese.
\newblock Fixed points for actions of {${\rm Aut}(F_n)$} on {$\rm CAT(0)$}
  spaces.
\newblock {\em M\"unster J. Math.}, 7(2):439--462, 2014.

\end{thebibliography}

\end{document}